\documentclass[12pt]{amsart}
\usepackage{amssymb,amsmath,amsthm,bm}
\usepackage[colorinlistoftodos,prependcaption,textsize=tiny]{todonotes}
 \definecolor{darkblue}{RGB}{0,0,160}
\usepackage{fouriernc} 
\usepackage[colorlinks=true,allcolors=darkblue]{hyperref}
\DeclareSymbolFont{usualmathcal}{OMS}{cmsy}{m}{n}
\DeclareSymbolFontAlphabet{\mathcal}{usualmathcal}

\usepackage[T1]{fontenc}

\usepackage{color}

\usepackage{parskip}
\usepackage{setspace}
\usepackage{amsmath,amsthm,amscd,amssymb, mathrsfs}

\usepackage{latexsym}

\numberwithin{equation}{section}

\theoremstyle{plain}
\newtheorem{theorem}{Theorem}[section]
\newtheorem{lemma}[theorem]{Lemma}

\newtheorem{proposition}[theorem]{Proposition}
\newtheorem{conjecture}[theorem]{Conjecture}

\theoremstyle{definition}

\newtheorem{example}[theorem]{Example}

\theoremstyle{remark}

\newtheorem{case[theorem]}{Case}

\title[\parbox{14cm}{\centering{A spherical extension theorem and applications \hspace{1in}}} \quad]{A spherical extension theorem and applications in positive characteristic}
\author{Doowon Koh and Thang Pham}

\address{Department of Mathematics\\
Chungbuk National University \\
Cheongju, Chungbuk 28644 Korea}
\email{koh131@chungbuk.ac.kr}

\address{University of Science\\ Vietnam National University, Hanoi
}
\email{thangpham.math@vnu.edu.vn}


\begin{document}
\begin{abstract} 
In this paper, we prove an extension theorem for spheres of square radii in $\mathbb{F}_q^d$, which improves a result obtained by Iosevich and Koh (2010). Our main tool is a new point-hyperplane incidence bound which will be derived via a cone restriction theorem due to the authors and Lee (2022). Applications on the distance problems will be also discussed. 
 \end{abstract}

\maketitle
\section{Introduction}
Let $q$ be an odd prime power, and $\mathbb{F}_q$ be a finite field of order $q$.  Let $\mathbb{F}_q^d$ be the $d$-dimensional vector space over $\mathbb{F}_q$. We endow the space $\mathbb F_q^d$ with  counting measure $dc$. We denote the dual space of $\mathbb{F}_q^d$ by $\mathbb F_{q*}^d$, and endow it with  normalized counting measure $dn$.  
For any algebraic variety $V$ in $\mathbb F_{q*}^d$, we will endow it with the normalized surface measure $d\sigma$ which is defined by the relation $d\sigma(x)=\frac{q^d}{|V|} 1_{V}(x) ~dn(x),$ where $|V|$ denotes the cardinality of $V.$ 

Let $\chi$ be a nontrivial additive character of $\mathbb F_q.$ For any function $g: \mathbb F_q^d \to \mathbb C,$ the Fourier transform of $g$ is defined by
$$ \widehat{g}(x):=\sum_{m\in \mathbb F_q^d} g(m) \chi(-x\cdot m).$$ 

This definition should be compared with the definition of the Fourier transform $\widehat{g}$ used in other papers. 
We emphasize that there does not appear a normalizing factor $q^{-d}$ in the definition of $\widehat{g},$ while such a normalizing factor has been used in many other articles such as \cite{IR06}, \cite{hart}, \cite{CEHIK10}, and \cite{covert}.

If $f$ is a complex-value function on the dual space, namely, $f: \mathbb F_{q*}^d \to \mathbb C$,  the inverse Fourier transform of $f$ is defined by
$$ f^{\vee} (m):= \frac{1}{q^d} \sum_{x\in \mathbb F_{q*}^d} f(x) \chi(m\cdot x).$$
In addition, the inverse Fourier transform of the measure $f d\sigma$  is defined by
$$ (fd\sigma)^{\vee} (m):= \frac{1}{|V|} \sum_{x\in V} f(x) \chi(m\cdot x).$$

Since there is an isomorphism between $\mathbb F_q^d$ and  its dual space $\mathbb F_{q*}^d$, for the sake of simplicity, we will simply write $\mathbb F_q^d$ for $\mathbb F_{q*}^d$, and in this paper, the only differences between the two spaces are corresponding measures. 

Let $P$ be the paraboloid in $\mathbb{F}_q^d$ defined by the equation $x_d=x_1^2+\cdots+x_{d-1}^2.$ For $\beta\in \mathbb{F}_q$, let $P_\beta$ be a\ translate of $P$ by $\beta$ defined by $x_d+\beta=x_1^2+\cdots+x_{d-1}^2.$ For $j\ne 0$, let $S_j$ be the sphere centered at the origin of radius $j$, namely,
  $$S_j=\{(x_1,\ldots, x_d)\in \mathbb F_q^d: x_1^2+ x_2^2 + \cdots + x_d^2=j\}.$$
Notice that  the definition of a radius in finite fields is different from  that in the Euclidean case.

 In this paper, the variety $V$ will be often considered as a sphere or the paraboloid $P$.

The $L^p\to L^r$ extension problem for the variety $V$ is to determine  all ranges of $1\le p, r\le \infty$ such that the following inequality 
\begin{equation}\label{defR}||(fd\sigma)^\vee||_{L^{r}(\mathbb{F}_q^d, dc)}\le C ||f||_{L^p(V, d\sigma)}\end{equation}
holds for all functions $f$ on $V$, where the positive constant $C$ does not depend on $q$. By duality,  the extension estimate \eqref{defR} implies that  one has
\begin{equation*}\label{dfff}\|\widehat{g}\|_{L^{p'}(V, d\sigma)} \le C \|g\|_{L^{r'}(\mathbb{F}_q^d, dc)}
\end{equation*}
for all functions $g$ on $(\mathbb F_q^d, dc)$, where $1/r+1/r'=1$ and $1/p+1/p'=1$. We will use the notation $R_{V}^*(p\to r)\ll 1$ to indicate that  the inequality \eqref{defR} holds. 

In this paper, we will use the following notation: $X \ll Y$ means that there exists  some absolute constant $C_1>0$ such that $X \leq C_1Y$, 
 and  $X\sim Y$ means $Y\ll X\ll Y$.

Necessary conditions for $R_V^*(p\to r)\ll 1$ can be given in terms of the size of $V$ and  the cardinality of an affine subspace $H$ lying on $V.$  Mockenhaupt and Tao  \cite{MT04}  indicated that if $V\subset \mathbb F_{q}^d$ with $|V|\sim q^{d-1}$ and $V$ contains an affine subspace $H$ with $|H|=q^k,$ then one has
\begin{equation*}
r\geq \frac{2d}{d-1}  \quad \mbox{and} \quad r\geq\frac{p(d-k)}{(p-1)(d-1-k)}.\end{equation*}
In recent years, there has been intensive progress in studying $L^2\to L^r$ and $L^p\to L^4$ extension estimates for spheres. More precisely, for the case of $R_{S_j}^*(2\to r)$, it is believed that 
\begin{equation}\label{ConjOdd}R_{S_j}^*(2\to r)\ll 1 \iff  r\ge \frac{2d+2}{d-1}\quad\mbox{for odd dimensions} ~~d\ge 3,\end{equation}
and
\begin{equation}\label{ConjEven} R_{S_j}^*(2\to r)\ll 1 \iff  r\ge  \frac{2d+4}{d}\quad\mbox{for even dimensions}~~ d\ge 2.\end{equation}

These conjectures can be derived by testing the inequality \eqref{defR} with $f=1_{H}$  and $f=1_{S_j}$, where $H$ denotes a maximal affine subspace lying in $S_j$. Moreover, it follows from \cite[Lemma 1.13]{KPV18} that $|H|=q^{\frac{d-1}{2}}$ for odd $d\ge 3,$ and $|H|=q^{\frac{d-2}{2}}$ for even $d\ge 2$ 

For all $d\ge 2$, Iosevich and Koh \cite{IK08} used Stein-Tomas argument to obtain an $L^2\to L^r$ extension result, which matches the conjecture \eqref{ConjOdd}. They also solved the extension conjecture \eqref{ConjEven} in two dimensions. For the case $d\ge 4$ even, we only know that
$$ R_{S_j}^*(2\to r)\ll 1\quad\mbox{for}~~ r\ge  \frac{2d+2}{d-1},$$
which is far from the conjecture \eqref{ConjEven}.  

Compared to the case of paraboloids, the estimates (\ref{ConjOdd}) and (\ref{ConjEven}) have been proved in \cite{MT04, IKL17}. Furthermore, in the case of paraboloids, if $d=4k+3$ and $-1$ is not a square, the conjecture is stronger, namely, $R_{P}^*\bigg(2\to \frac{2d+6}{d+1}\bigg)\ll 1$. The best current estimate is $R_P^*\bigg(2\to \frac{2d+4}{d}\bigg)\ll 1$ due to the authors and Vinh \cite{KPV18}. It is known for the paraboloid case that the sharp $L^2\to L^r$ extension result can be derived by using the additive energy of a set on a paraboloid, for example, see \cite{Le13, Le19, Le20, RS19}. However, such a connection is not known for spheres.  Moreover, the spherical extension problem is harder than the paraboloid case, since the Fourier transform of non zero-radii spheres is reduced to the Kloosterman sum whose explicit form is not known yet.

If $-1$ is not a square and $d=4k+2$, Iosevich, Lee, Shen, and the authors \cite{pham} proved that the conjecture (\ref{ConjEven}) holds for the sphere of zero radius. The main difference between the zero radius and non-zero radius spheres is that we can use the Gauss sum in the place of the Kloosterman sum in the Fourier decay. In addition, the explicit form of the Gauss sum is very well-known, for instance, see \cite{LN97}. 

Similarly, the $L^p\to L^4$ extension problem for the sphere $S_j$ is to determine  all ranges of $p$ such that the following inequality 
\begin{equation*}\label{defR1}||(fd\sigma)^\vee||_{L^{4}(\mathbb{F}_q^d, dc)}\le C ||f||_{L^p(S_j, d\sigma)}\end{equation*}
holds for all functions $f$ on $S_j$, where the positive constant $C$ does not depend on $q$. 

In odd dimensional spaces, it is known in \cite[Section 2]{I-K} that the Stein-Tomas exponent toward $L^p\to L^4$ cannot be improved in general. This comes from the fact that if $q\equiv 1\mod 4$, $d$ is odd, and the sphere $S_j$ is of non-zero square radius, then $S_j$ contains an affine subspace of dimension $\frac{d-1}{2}$. We refer the interested reader to \cite{I-K} for more details. 

In a recent paper, the authors and Vinh \cite{KPV18} showed that when we study spheres of primitive radii,  the Stein-Tomas exponent toward $L^p\to L^4$ can be considerably improved. More precisely, the following results have been obtained in \cite{KPV18}.
\begin{theorem}[\cite{KPV18}]\label{thm:main3'}
Let $g$ be a primitive element in $\mathbb{F}_q.$ If either $d=4k+1$  with $k\in \mathbb N$ or $d=4k-1$ and $q\equiv 1\mod 4$, then we have
\[R^*_{S_g}\left(\frac{4d}{3d-2}\to 4\right)\ll 1.\]
\end{theorem}

The most interesting aspect of these theorems comes from the view that there is a different extension phenomenon between paraboloids and spheres. More precisely, in the case of the paraboloid $P$, we know from \cite{MT04} that 
if  $d=4k+1$, $k\in \mathbb{N}$ or $d=4k-1$ with $q\equiv 1 \mod 4,$ then
 $$R_P^*(p\to 4)\ll 1 \quad \mbox{if and only if }\quad \frac{4d-4}{3d-5} \le p\le \infty.$$
These estimates are optimal. Therefore, under the same conditions on $q$ and $d$, the spherical extension theorems are much better. Based on dimensions of affine subspaces, the following conjecture has been provided in \cite{KPV18}. 

\begin{conjecture}\label{conj1.3} Let $S_j$ be the sphere with non-zero radius in $\mathbb F_q^d.$   The following statements hold.
\begin{enumerate}
\item If $d=4k+1$, $k\in \mathbb N,$ and $j$ is not square, then the bound $R^*_{S_j}\left(\frac{4d+4}{3d+1} \to 4\right)\ll 1$ gives the sharp $L^p\to L^4$ estimate.
\item If $d=4k-1$, $k\in \mathbb N,$ $q\equiv 1 \mod{4}$, and $j$ is not square, then the bound $R^*_{S_j}\left(\frac{4d+4}{3d+1} \to 4\right)\ll 1$ gives the sharp $L^p\to L^4$ estimate.
\item If $d=4k-1$, $k\in \mathbb N,$ $q\equiv 3 \mod{4},$ and $j$ is square, then  the bound $R^*_{S_j}\left(\frac{4d+4}{3d+1} \to 4\right)\ll 1$ gives the sharp $L^p\to L^4$ estimate.
\end{enumerate} 
\end{conjecture}

In the proofs of Theorem \ref{thm:main3'}, the main tool in \cite{KPV18} is \textit{the first association scheme graph}, which works for spheres of primitive radii. However, in the case (3) of  Conjecture \ref{conj1.3}, namely, when the radius of the sphere is a square, we get nothing from that method. The main purpose of this paper is to address that case. Our main tool is a new point-hyperplane incidence bound in $\mathbb{F}_q^d$, which will be derived via a cone restriction theorem. Our main result is the following. 

\begin{theorem}\label{thm:main1}
Suppose that $d=4k-1$ with  $k\in \mathbb{N}$, and $q\equiv 3\mod 4$. Let $S_j$ be a sphere in $\mathbb{F}_q^d$ of square radius $j\ne 0$ . Then, we have
\begin{equation}\label{CThm1}R^*_{S_j}\left(\frac{4d}{3d-2}\to 4\right)\ll 1.\end{equation}
\end{theorem}

Although Theorems \ref{thm:main3'} and \ref{thm:main1} do not match the conjecture, in even dimensions, it has been proved in \cite{pham} that the same estimate is sharp. 
\bigskip
\paragraph{\textbf{Applications:}} We now discuss an application of $L^p\to L^4$ extension estimates to a distance problem. 

Given two points $x=(x_1, \ldots, x_d)$ and $y=(y_1, \ldots, y_d)$ in $\mathbb{F}_q^d$, the distance function between them is defined by the following equation
\[||x-y||=(x_1-y_1)^2+\cdots+(x_d-y_d)^2.\]
For $A\subset \mathbb{F}_q^d$, we denote the set of distances determined by pairs of points in $A$ by $\Delta_2(A)$, namely, 
\[\Delta_2(A):=\{||x-y||\colon x, y\in A\}.\]
The Erd\H{o}s-Falconer distance problem over finite fields asks for the smallest number $N$ such that  for any $A\subset \mathbb{F}_q^d$ with $|A|\ge C q^{N}$ the distance set $\Delta_2(A)$ covers the whole field $\mathbb{F}_q$ or a positive proportion of all possible distances. 

In 2007, Iosevich and Rudnev \cite{IR06} proved that, for any $A\subset \mathbb{F}_q^d$ with $d\ge 2$, if the size of $A$ is at least $4q^{(d+1)/2}$, then $\Delta_2(A)=\mathbb{F}_q$. It was proved in \cite{hart} that the exponent $\frac{d+1}{2}$ is sharp in odd dimensional vector spaces in the sense that  for every $\epsilon>0$, there exist a constant $c>0$ and a set $A\subset \mathbb{F}_q^d$ with $d\ge 3$ odd such that $|A|\ge cq^{\frac{d+1}{2}-\epsilon}$ and $|\Delta_2(A)|\sim q^{1-\epsilon}$. In even dimensions, it has been conjectured that the right exponent should be $d/2$, which is in line with the Falconer distance conjecture in the continuous setting \cite{falconer}, which is still wide open.  In two dimensional vector spaces over finite fields, the best current exponent is $\frac{4}{3}$ for arbitrary finite fields \cite{CEHIK10} and $\frac{5}{4}$ for prime fields \cite{mu}. 

For $A\subset \mathbb{F}_q^d$, we define 
\[\Delta_3(A):=\{||x_1\pm x_2\pm x_3||\colon x_1, x_2, x_3\in A\}.\]
Since any choice of the signs will not play an important role in deducing our results, we  choose  the  positive signs for simple writing in the definition of  $\Delta_3(A):$
\[\Delta_3(A):=\{||x_1+ x_2+ x_3||\colon x_1, x_2, x_3\in A\}.\]

Covert, Koh, and Pi \cite{covert2} studied the following variant of the Erd\H{o}s-Falconer distance problem: For $A\subset \mathbb{F}_q^d$, how large does $A$ need to be such that $\Delta_3(A)$ covers the whole field $\mathbb{F}_q$ or at least a positive proportion of all elements? 

The geometric meaning of the size of $\Delta_3(A)$ can be considered as the norm of thriple sums of elements of $A.$

It has been indicated in \cite{covert2} that in odd dimensions, in order to obtain $|\Delta_3(A)|\gg q$, one must have $|A|\gg q^{\frac{d+1}{2}}$, but in even dimensions, we can decrease the threshold $q^{\frac{d+1}{2}}$ to $q^{\frac{d+1}{2}-\epsilon_d}$ for some $\epsilon_d=\epsilon(d)>0$. More precisely, they proved the following theorem.

\begin{theorem}\label{caothien}
Let $A$ be a set in $\mathbb{F}_q^d$ with $d$ even. 
\begin{enumerate}
\item Suppose that $d=4$ and $|A|\gg q^{\frac{32}{13}}$, then we have  $|\Delta_3(A)|\gg q.$
\item Suppose that $d\ge 6$ and $|A|\gg q^{\frac{d+1}{2}-\frac{1}{9d-18}+\epsilon}$ for any $\epsilon>0$, then we have $|\Delta_3(A)|\gg q.$
\end{enumerate}

\end{theorem}

The most interesting aspect of this result is that they have made a connection between the size of $\Delta_3(A)$ and $L^4$ extension estimates for spheres in $\mathbb{F}_q^d$. We now take advantage of the sharp $L^4$ estimate for spheres in even dimensions to improve Theorem \ref{caothien}. Our improvement is as follows.

\begin{theorem}\label{thm:main4-thang}
Let $A$ be a set in $\mathbb{F}_q^d$ with $d$ even. 
\begin{enumerate}
\item Suppose that $d=4$ and $|A|\gg q^{\frac{12}{5}}$, then we have 
\[|\Delta_3(A)|=\#\{||x_1+x_2+x_3||\colon x_1, x_2, x_3\in A\}\gg q.\]
\item Suppose that $d\ge 6$ and $|A|\gg q^{\frac{d+1}{2}-\frac{1}{3d-4}}$, then we have 
\[|\Delta_3(A)|=\#\{||x_1+x_2+x_3||\colon x_1, x_2, x_3\in A\}\gg q.\]
\end{enumerate}
\end{theorem}
\section{A new incidence theorem}
Let $\mathcal{P}$ be a set of points in $\mathbb{F}_q^d$ and $\Pi$ be a set of hyperplanes in $\mathbb{F}_q^d$. Let $I(\mathcal{P}, \Pi)$ be the number of incidences between $\mathcal{P}$ and $\Pi$, i.e.
\[I(\mathcal{P}, \Pi)=\#\left\lbrace (p, \pi)\in \mathcal{P}\times \Pi\colon p\in \pi\right\rbrace.\]
It is well-known that 
\begin{equation}\label{universal}\left\vert I(\mathcal{P}, \Pi)-\frac{|\mathcal{P}||\Pi|}{q}\right\vert\le q^{\frac{d-1}{2}}|\mathcal{P}|^{1/2}|\Pi|^{1/2}.\end{equation}
A proof can be found in \cite{vinhline} or \cite{line} in the language of block designs from the 1980s. 

For incidence bounds of this type, the value $\frac{|\mathcal{P}||\Pi|}{q}$ is understood as the expected number or the main term, and  the value $q^{\frac{d-1}{2}}|\mathcal{P}|^{1/2}|\Pi|^{1/2}$ is the error term. There are several examples that show that the error term cannot be improved. Let us have a brief discussion here. 

\begin{example} \label{Ex2.1}
Assume that $q\equiv 1\mod 4$ and $d$ is odd, there exist sets $\mathcal{P}$ and $\Pi$ such that $|\mathcal{P}|=|\Pi|=q^{\frac{d-1}{2}}$ and $I(\mathcal{P}, \Pi)=q^{\frac{d-1}{2}}|\mathcal{P}|^{1/2}|\Pi|^{1/2}$. Indeed, it has been proved in \cite{hart} that if either $d$ is even and $q\equiv 1\mod 4$ or $d-1=4k$, then there exist $\frac{d-1}{2}$ vectors $\{v_1, \ldots, v_{\frac{d-1}{2}}\}$ in $\mathbb{F}_q^{d-1}$, which are linearly independent, and $v_i\cdot v_j=0$ for all $i, j$. Let $A=\mathtt{Span}\bigg(v_1, \ldots, v_{\frac{d-1}{2}}\bigg)\subset\mathbb{F}_q^{d-1}$. Then we have $|A|=q^{\frac{d-1}{2}}$. Given $\lambda\in \mathbb{F}_q\setminus \{0\}$, define $\mathcal{P}=A\times \{\lambda\}$ and $\Pi$ being the set of hyperplanes defined by the equation $a_1x_1+\cdots+a_{d-1}x_{d-1}+\lambda x_d=\lambda^2$ with $(a_1, \ldots, a_{d-1})\in A$. Since $||a-b||=0$ and $||a||=||b||=0$ for all $a, b\in A$,  we have $a\cdot b=0.$ Hence, the number of incidences between $\mathcal{P}$ and $\Pi$ is $|\mathcal{P}||\Pi|=q^{\frac{d-1}{2}}|\mathcal{P}|^{1/2}|\Pi|^{1/2}$.
\end{example}
In the same argument of Example \ref{Ex2.1}, we also have the following.
\begin{example}
Assume that $q\equiv 3\mod 4$ and $d=4k+1$, $k\in \mathbb{N}.$ Then  there exist sets $\mathcal{P}$ and $\Pi$ with $|\mathcal{P}|=|\Pi|=q^{\frac{d-1}{2}}$ such that $I(\mathcal{P}, \Pi)=q^{\frac{d-1}{2}}|\mathcal{P}|^{1/2}|\Pi|^{1/2}$. 
\end{example}

The main purpose of this section is to provide an improvement of the estimate (\ref{universal}) in the case the point set $\mathcal{P}$ is distributed in at most $q^{1-\epsilon}$ spheres or translates of the paraboloid $P$ for some $0< \epsilon<1$. 

\begin{theorem}\label{newincidence}
Let $\mathcal{P}$ be a set of points in $\mathbb{F}_q^d$ and $\Pi$ be a set of hyperplanes in $\mathbb{F}_q^d$. Let $t$ be the minimum number of spheres of square radii (or translates of the paraboloid $P$) that cover the set $\mathcal{P}$. We assume in addition that $q\equiv 3\mod 4$ and $d=4k-1$ with $k\in \mathbb{N}$.  Then the number of incidences between $\mathcal{P}$ and $\Pi$ satisfies
\[\left\vert I(\mathcal{P}, \Pi)-\frac{|\mathcal{P}||\Pi|}{q}\right\vert\ll t^{1/2}q^{\frac{d-2}{2}}|\mathcal{P}|^{1/2}|\Pi|^{1/2}+t^{1/2}q^{\frac{d-3}{4}}|\mathcal{P}|^{1/2}|\Pi|.\]
\end{theorem}
Theorem \ref{newincidence} is most effective when $t$ is bounded by a constant number. 
\section{Proof of Theorem \ref{newincidence}}
To prove Theorem \ref{newincidence}, we recall the following result from \cite[Lemma 4.1]{KLP}.
\begin{lemma}\label{lm1}
For $n\in \mathbb{N}$, let $C_n$ be the cone in $\mathbb{F}_q^n$ defined by 
\[C_n:=\left\lbrace m\in \mathbb{F}_q^n\colon m_n^2=m_1^2+\cdots+m_{n-1}^2\right\rbrace.\]
Suppose that $n\equiv 0 \mod 4$ and $q\equiv 3\mod 4$, then, for any $G\subset \mathbb{F}_q^n$, we have 
$$\|\widehat{G}\|_{L^2(C_n, d\sigma)} 
\ll |G|^{1/2}+\frac{|G|}{q^{\frac{n}{4}}}.
$$
\end{lemma}

We are ready to prove Theorem \ref{newincidence}.
\begin{proof}[Proof of Theorem \ref{newincidence}]
We first prove the case $t=1$. We consider two following cases:

{\bf \textbf{Case} $1$:}
 Assume that $\mathcal{P}$ lies on a sphere centered at the origin of radius $r=u^2$ for some $u\in \mathbb{F}_q\setminus \{0\}$. We can assume the center of the sphere is the origin since the number of incidences is invariant under translations. Notice that  a hyperplane in $\Pi$, given by the equation $a\cdot x=a_{d+1}$,  can be identified with   a vector $(a, a_{d+1})$ in $\mathbb F_q^d$ $\times {\mathbb F_q}={\mathbb F_q^{d+1}}.$

 Define 
\[\mathcal{P}':=\left\lbrace (\lambda p, \lambda u)\colon p\in \mathcal{P}, \lambda\in \mathbb{F}_q\right\rbrace\subset \mathbb{F}_q^{d+1},\] and \[\Pi':=\left\lbrace s(a_1, \ldots, a_d,  -u^{-1}\cdot a_{d+1})\colon s \in \mathbb{F}_q^*,~ a_1x_1+\cdots+a_dx_d=a_{d+1}\in \Pi\right\rbrace.\]
It is clear that $|\mathcal{P}'|\le q|\mathcal{P}|$ and $|\Pi'|\le q|\Pi|$. Note that $\mathcal{P}'$ is a set on the cone $C_{d+1}$. 

We have 
$$I(\mathcal{P}, \Pi) = \sum_{p\in \mathcal{P},  (a, a_{d+1})\in \Pi: a\cdot p=a_{d+1}} 1.$$
The condition $a\cdot p=a_{d+1}$  is equivalent with the equation   $(p, u)\cdot (a, -u^{-1} a_{d+1})=0,$ which implies that  $(\lambda p, \lambda u) \cdot (a, -u^{-1} a_{d+1})=0$ for all $\lambda \in \mathbb F_q.$  Hence,  
$$  I(\mathcal{P}, \Pi) = q^{-1}\sum_{x \in {\mathcal{P}'} ,  (a, a_{d+1})\in \Pi: ~x \cdot (a, -u^{-1} a_{d+1})=0} 1.$$
Applying the orthogonality property of $\chi$,  the incidence $I(\mathcal{P}, \Pi)$ becomes
$$  \frac{|\mathcal{P'}||\Pi|}{q^2} + q^{-2} \sum_{ x \in {\mathcal{P}'} ,  (a, a_{d+1})\in \Pi} \sum_{s\ne 0} \chi(x \cdot s(a, -u^{-1}a_{d+1})).$$ 

From the definition of $\Pi'$,  we get
$$I(\mathcal{P}, \Pi)=\frac{|\mathcal{P}||\Pi|}{q}+\frac{1}{q^2}\sum_{x\in \mathcal{P}', y\in \Pi'}\chi(x\cdot y).$$ By the Cauchy-Schwarz inequality,  

\begin{equation}\label{3.11} \left|I (\mathcal{P}, \Pi)-\frac{|\mathcal{P}||\Pi|}{q}\right| \le \frac{1}{q^2}\sum_{x\in \mathcal{P}'} |\widehat{\Pi'}(x)|\le\frac{1}{q^2}  |\mathcal{P'}|^{\frac{1}{2}} \left( \sum_{x\in C_{d+1}}|\widehat{\Pi'}(x)|^2\right)^{\frac{1}{2}}.\end{equation}


Using Lemma \ref{lm1} with $G=\Pi'$ and $n=d+1$, we have 
\[\sum_{x\in C_{d+1}}|\widehat{\Pi'}(x)|^2\ll q^d\cdot \left( |\Pi'|+q^{-\frac{d+1}{2}}|\Pi'|^2\right).\]
Substituting this estimate into (\ref{3.11}) gives us 
\[\left\vert I(\mathcal{P}, \Pi)-\frac{|\mathcal{P}||\Pi|}{q}\right\vert\ll q^{\frac{d-2}{2}}|\mathcal{P}|^{1/2}|\Pi|^{1/2}+q^{\frac{d-3}{4}}|\mathcal{P}|^{1/2}|\Pi|.\]

{\bf Case $2$:} Assume that $\mathcal{P}$ lies on $P_\beta$, recall, a translate of $P$ by $\beta$ defined by $x_d+\beta=x_1^2+\cdots+x_{d-1}^2$. 
Without loss of generality, we may assume that $\beta=0$ since the number of  incidences is not changed by  translations.
 Define 
\[\mathcal{P}':=\left\lbrace (\lambda p, \lambda )\colon p\in \mathcal{P}, \lambda\in \mathbb{F}_q\right\rbrace\subset \mathbb{F}_q^{d+1},\] and \[\Pi':=\left\lbrace s(a_1, \ldots, a_d, -a_{d+1})\colon s\in \mathbb{F}_q^*, ~a_1x_1+\cdots+a_dx_d=a_{d+1}\in \Pi\right\rbrace.\]

Since $\mathcal{P}$ lies on the paraboloid $P$ in $\mathbb F_q^d$, one can check that $\mathcal{P}'$ lies on the variety defined by the equation 
\[x_{d+1}\cdot x_d=x_1^2+\cdots +x_{d-1}^2,\]
which is the cone $C_{d+1}$ after a change of variables. Thus, the argument in the case $1$ will give us the desired exponent. 

In other words, the case $t=1$ is proved. 

For $t>1$, one can partition the point set $\mathcal{P}$ into $t$ subsets, each on a variety, then we can apply the case $t=1$ for each of these subsets, and the theorem follows from the Cauchy-Schwarz inequality. 
\end{proof}

\section{An energy bound}
 For a set $A$ in $\mathbb F_q^d,$  the additive energy, denoted by $E(A)$ is defined as the number of the pairs $(a,b,c,d)\in A^4$ such that $a+b=c+d.$ 

\begin{theorem}\label{energy-sphere}
Suppose that $d=4k-1$ and $q\equiv 3\mod 4$. Let $S_j$ be a sphere of square radius $j\ne 0$ in $\mathbb{F}_q^d$.  For $A\subset S_j$, we have
\[E(A)\ll \frac{|A|^3}{q}+q^{\frac{d-2}{2}}|A|^2+q^{\frac{d-3}{4}}|A|^{5/2}.\]
\end{theorem}
\begin{proof}
We start with the following observation. Given $a, b, c, d\in S_j$, if $a+b=c+d$, then we have 
\[(b-d)\cdot (a-d)=0.\]
This can be viewed as a right angle at $d$. Thus $E(A)$ is bounded by the number of triples $(a, b, d)\in A^3$ such that $(b-d)\cdot (a-d)=0$. We now fall into two cases: 

{\bf Case $1$:} Let $E_1$ be the number of triples $(a, b, d)\in A^3$ such that   $||a-d||=0$ or $||b-d||=0$. We are going to show that 
\[E_1\ll \frac{|A|^3}{q}+q^{\frac{d-2}{2}}|A|^2+q^{\frac{d-3}{4}}|A|^{5/2}.\]
Indeed, if $||a-d||=0$, then we have $a\cdot d=j$. The identity $a\cdot d=j$ can be understood as an incidence between the point $d\in A\subset S_j$ and the hyperplane defined by the equation $a\cdot x=j$. One can apply Theorem \ref{newincidence} with $t\ll 1$ to show that the number of pairs $(a, b)\in A\times A$ such that $||a-d||=0$ is bounded by 
\[\frac{|A|^2}{q}+O(q^{\frac{d-2}{2}}|A|+q^{\frac{d-3}{4}}|A|^{3/2}).\]
By the same argument,  this estimate holds in the case when $||b-d||=0.$
Thus, 
\[E_1\ll \frac{|A|^3}{q}+q^{\frac{d-2}{2}}|A|^2+q^{\frac{d-3}{4}}|A|^{5/2}.\]
{\bf Case $2$:} Let $E_2$ be the number of triples $(a, b, d)\in A^3$ such that $||a-d||\ne 0, ||b-d||\ne 0$. 

We actually prove  the case when at least one of $||a-d||$ and $||b-d||$  is not zero, which  clearly contains the assumption of Case 2. 

For a fixed $d\in A$, we now count the number of pairs $(a, b)\in A^2$ such that $(a-d)\cdot (b-d)=0$. 
When $||a-d||\ne 0$, there is no other point $a'\in A$ such that $a'-d=\lambda (a-d)$ for some $\lambda \in \mathbb{F}_q^*\setminus\{1\}$. The same also holds for $b-d$. 

We observe that the identity $(a-d)\cdot (b-d)=0$ is equivalent with an incidence between the point $a\in A$ and the hyperplane defined by $(b-d)\cdot x=(b-d)\cdot d$. Since $A\subset S_j$, Theorem \ref{newincidence} with $t\ll 1$  gives us that the number of pairs $(a, b)\in A^2$ such that $(a-d)\cdot (b-d)=0$ is at most 
\[\frac{|A|^2}{q}+O(q^{\frac{d-2}{2}}|A|+q^{\frac{d-3}{4}}|A|^{3/2}).\]
Taking the sum over all $d\in A$, 
\[E_2\ll \frac{|A|^3}{q}+q^{\frac{d-2}{2}}|A|^2+q^{\frac{d-3}{4}}|A|^{5/2}.\]



Putting $E_1$ and $E_2$ together, the theorem follows. $\square$
\end{proof}

\section{Extension theorems for spheres}
In this section, we give a complete proof of Theorem \ref{thm:main1}. We start this section with the following lemma.

\begin{lemma}\label{Adidaphat3}
Suppose that $d=4k-1$ and $q\equiv 3\mod 4$. Let $S_j$ be a sphere of square radius $j\ne 0$ in $\mathbb{F}_q^d$.  For $A\subset S_j$ of size $n$, we have
$$ \|(Ad\sigma)^\vee\|_{L^{4}(\mathbb F_q^d, dc)} 
\ll \left\{  \begin{array}{ll} n^{\frac{3}{4}}q^{\frac{-3d+3}{4}}  \quad&\mbox{for}~~q^{\frac{d+1}{2}} \le n\ll q^{d-1}\\
n^{\frac{5}{8}}q^{\frac{-11d+13}{16}} \quad&\mbox{for}~~ q^{\frac{d-1}{2}} \le n\le q^{\frac{d+1}{2}}\\
n^{\frac{1}{2}}q^{\frac{-5d+6}{8}} \quad&\mbox{for}~~ q^{\frac{d-2}{2}} \le n\le q^{\frac{d-1}{2}}\\
n^{\frac{3}{4}}q^{\frac{-3d+4}{4}} \quad&\mbox{for}~~ 1 \le n\le q^{\frac{d-2}{2}}.\end{array}\right.
$$
\end{lemma}
\begin{proof} Using the orthogonal property of $\chi$, we have
\[ \|(Ad\sigma)^\vee\|_{L^{4}(\mathbb F_q^d, dc)} =\frac{q^{\frac{d}{4}}}{|S_j|}\cdot  E(A)^{1/4} \sim q^{\frac{-3d+4}{4}}E(A)^{1/4}.\] 
We now fall into two cases:

{\bf Case $1$:}
 If $q^{\frac{d-2}{2}}\le n\ll q^{d-1}$, then we can apply Theorem \ref{energy-sphere} to get the desired bounds. 

{\bf Case $2$:} If $n\le q^{\frac{d-2}{2}}$, then we use the trivial bound $n^3$ for the energy to conclude the proof.
\end{proof}
\paragraph{\bf{Proof of Theorem} \ref{thm:main1}}
We begin by observing that  the conclusion of Theorem \ref{thm:main1} in \eqref{CThm1} is the same as 
\begin{equation*}\label{sphere1} \|(fd\sigma)^\vee\|_{L^{4}(\mathbb F_q^d, dc)} \ll \|f\|_{L^{4d/(3d-2)}(S_j, d\sigma)}\sim \left(q^{-d+1} \sum_{x\in S_j} |f(x)|^{\frac{4d}{3d-2}}\right)^{\frac{3d-2}/{4d}},\end{equation*}
since the inequality $\ll$ above follows from the definition of \eqref{CThm1} and the similarity symbol $\sim$ above is obtained by  using  the definition of $\|f\|_{L^{4d/(3d-2)}(S_j, d\sigma)}$ together with the fact that $|S_j|\sim q^{d-1}.$
Thus,  to complete the proof of Theorem \ref{thm:main1}, it suffices to prove the following inequality:
\[q^{\frac{3d^2-5d+2}{4d}}\|(fd\sigma)^\vee\|_{L^{4}(\mathbb F_q^d, dc)}\ll \left(\sum_{x\in S_j} |f(x)|^{\frac{4d}{3d-2}}\right)^{\frac{3d-2}{4d}}.\]
Without loss of generality, we may assume that the test function $f$ is a nonnegative real valued function since a general complex valued function $f$ is written as the form $f_1 + i f_2 $ for some real valued functions $f_1$ and $f_2,$ and a real valued function $f$ can be expressed as the difference of two nonnegative real valued functions. 
Furthermore, normalizing the function $f$ if necessary, we may assume that 
\begin{equation}\label{Adidaphat5}\sum_{x\in S_j}|f(x)|^{\frac{4d}{3d-2}}=1.\end{equation}
Therefore, it is sufficient to show that 
\[T:=q^{\frac{3d^2-5d+2}{4d}} \|(fd\sigma)^\vee\|_{L^{4}(\mathbb F_q^d, dc)} \ll 1.\]
Notice that  for a nonnegative real valued function $f$, 
$$\|(fd\sigma)^\vee\|_{L^{4}(\mathbb F_q^d, dc)} = \frac{q^{\frac{d}{4}}}{|S_j|} \left( \sum_{a,b,c, d\in S_j: a+b=c+d} f(a)f(b)f(c)f(d)\right)^{1/4}.$$ 
Hence, without loss of generality,  we may assume that 
the test function $f$ takes the following form:
\begin{equation}\label{Adidaphat4} f(x)=\sum_{i=0}^{\infty} 2^{-i} {A_i(x)},\end{equation}
where  $\{A_i\}$ are disjoint subsets of $S_j.$
It follows from \eqref{Adidaphat5} and \eqref{Adidaphat4} that  
$$ \sum_{i=0}^{\infty} 2^{-\frac{4d}{3d-2}i} |A_i| =1,$$ 
which gives us
\begin{equation}\label{Adidaphat6} |A_i|\le  2^{\frac{4d}{3d-2}i},\quad~~ \forall i.\end{equation}
Let $N=C \log{q}$,  a positive integer for some sufficiently large constant $C.$ It follows that
\begin{align*} T &\le q^{\frac{3d^2-5d+2}{4d}} \sum_{i=0}^N 2^{-i}\|(A_i d\sigma)^\vee\|_{L^{4}(\mathbb F_q^d, dc)} + q^{\frac{3d^2-5d+2}{4d}} \sum_{i=N+1}^\infty 2^{-i}\|(A_i d\sigma)^\vee\|_{L^{4}(\mathbb F_q^d, dc)}\\
&=: M +R.\end{align*}
We first bound $R.$  Since $ |(A_i d\sigma)^\vee (m)|\le 1$ for all $m\in \mathbb F_q^d,$  it is clear that  $\|(A_i d\sigma)^\vee\|_{L^{4}(\mathbb F_q^d, dc)}\le q^{d/4}.$ It therefore follows that
$$ R\le q^{\frac{3d^2-5d+2}{4d}}  q^{\frac{d}{4}}\sum_{i=N+1}^\infty 2^{-i} \ll  q^{\frac{4d^2-5d+2}{4d}} 2^{-N} \ll 1.$$
We now estimate $M.$ To do this,  we decompose  the sum $\sum_{i=0}^N$ as four subsummands as follows:
\begin{align*} \sum_{i=0}^N &= \sum_{\substack{0 \le i \le N\\ 1\le 2^{\frac{4d}{3d-2}i} \le  q^{\frac{d-2}{2}}}}  + 
\sum_{\substack{0 \le i \le N\\ q^{\frac{d-2}{2}} \le 2^{\frac{4d}{3d-2}i} \le  q^{\frac{d-1}{2}}}} +\sum_{\substack{0 \le i \le N\\ q^{\frac{d-1}{2}}\le 2^{\frac{4d}{3d-2}i} \le  q^{\frac{d+1}{2}}}} +\sum_{\substack{0 \le i \le N\\ q^{\frac{d+1}{2}}\le 2^{\frac{4d}{3d-2}i} \ll  q^{d-1}}} \\
&=: \sum_1 +\sum_2 +\sum_{3} +\sum_{4}.\end{align*}
Then, with notations above, the term $M$ is written by
\begin{align*}M=&  q^{\frac{3d^2-5d+2}{4d}}\sum_{1} 2^{-i} \|({A_i}d\sigma)^\vee\|_{L^{4}(\mathbb F_q^d, dc)} 
 +  q^{\frac{3d^2-5d+2}{4d}}\sum_{2} 2^{-i} \|({A_i}d\sigma)^\vee\|_{L^{4}(\mathbb F_q^d, dc)}\\
 &+q^{\frac{3d^2-5d+2}{4d}}\sum_{3} 2^{-i} \|({A_i}d\sigma)^\vee\|_{L^{4}(\mathbb F_q^d, dc)}+q^{\frac{3d^2-5d+2}{4d}}\sum_{4} 2^{-i} \|({A_i}d\sigma)^\vee\|_{L^{4}(\mathbb F_q^d, dc)}\\
 &=: M_1 + M_2+M_3+M_4.
 \end{align*} 
Employing Lemma \ref{Adidaphat3} and using \eqref{Adidaphat6},  we get
\[ M_1 \ll q^{\frac{-d+2}{4d}} \sum_{\substack{0 \le i \le N\\ 1\le 2^{\frac{4d}{3d-2}i} \le  q^{\frac{d-2}{2}}}} 2^{-i} |A_i|^{\frac{3}{4}} 
\le q^{\frac{-d+2}{4d}} \sum_{\substack{0 \le i \le N\\ 1\le 2^{\frac{4d}{3d-2}i} \le q^{\frac{d-2}{2}}}} 2^{\frac{2}{3d-2}i} \ll q^{\frac{-d+2}{4d}} \cdot q^{\frac{d-2}{4d}}=1,\]
\[ M_2 \ll q^{\frac{d^2-4d+4}{8d}}  \sum_{\substack{0 \le i \le N\\ q^{\frac{d-2}{2}} \le 2^{\frac{4d}{3d-2}i} \le  q^{\frac{d-1}{2}}}} 2^{-i} |A_i|^{\frac{1}{2}} 
\le q^{\frac{d^2-4d+4}{8d}}  \sum_{\substack{0 \le i\le N\\ q^{\frac{d-2}{2}} \le 2^{\frac{4d}{3d-2}i} \le  q^{\frac{d-1}{2}}}} 2^{\frac{-d+2}{3d-2}i} \ll q^{\frac{d^2-4d+4}{8d}} \cdot q^{\frac{-d^2+4d-4}{8d}} =1,\]
and 
\[M_4 \ll q^{\frac{-d+1}{2d}}\sum_{\substack{0 \le i \le N\\ q^{\frac{d+1}{2}}\le 2^{\frac{4d}{3d-2}i} \ll  q^{d-1}}} 2^{-i} |A_i|^{\frac{3}{4}} 
\le  q^{\frac{-d+1}{2d}}\sum_{\substack{0 \le i \le N\\ q^{\frac{d+1}{2}}\le 2^{\frac{4d}{3d-2}i} \ll  q^{d-1}}} 2^{\frac{2}{3d-2}i} \ll  q^{\frac{-d+1}{2d}}\cdot q^{\frac{d-1}{2d}} =1. \]
It remains to show that  $M_3 \ll 1,$ which will be proven separately in the cases of $d=3$ and  $d=4k-1\ge 7$.
As before, it follows by Lemma \ref{Adidaphat3} and \eqref{Adidaphat6}
that 
\[ M_3 \ll q^{\frac{d^2-7d+8}{16d}}  \sum_{\substack{0 \le i \le N\\ q^{\frac{d-1}{2}} \le 2^{\frac{4d}{3d-2}i} \le  q^{\frac{d+1}{2}}}} 2^{-i} |A_i|^{\frac{5}{8}} 
\le q^{\frac{d^2-7d+8}{16d}}  \sum_{\substack{0 \le i\le N\\ q^{\frac{d-1}{2}} \le 2^{\frac{4d}{3d-2}i} \le  q^{\frac{d+1}{2}}}} 2^{\frac{-d+4}{6d-4}i}. \]

Hence, if $d=3$, then 
$$M_3 \ll q^{-\frac{1}{12}}  \sum_{\substack{0 \le i\le N\\ q \le 2^{\frac{12}{7}i} \le  q^{2}}} 2^{\frac{i}{14}} \ll q^{-\frac{1}{12}} q^{\frac{1}{12}}=1. $$
On the other hand, if  $d\ge 7,$ then we have
$$M_3 \ll q^{\frac{d^2-7d+8}{16d}}  \sum_{\substack{0 \le i\le N\\ q^{\frac{d-1}{2}} \le 2^{\frac{4d}{3d-2}i} \le  q^{\frac{d+1}{2}}}} 2^{\frac{-d+4}{6d-4}i}
\ll q^{\frac{d^2-7d+8}{16d}}  q^{\frac{-d^2+5d-4}{16d}} = q^{\frac{-2d+4}{16d}}\le 1.$$
This completes the proof of the theorem. $\square$
\section{Three-distance problem (Theorem \ref{thm:main4-thang})}

To prove Theorem \ref{thm:main4-thang}, we need the following results. The first proposition is known as the interpolation proposition.  A detailed proof can be found in \cite{aka}. 

\begin{proposition}\label{co2} Let $1\le  r_0, r_1, p_0, p_1\le \infty$ with $r_0\le r_1$ and $p_0\le p_1$.
\begin{enumerate}
\item Suppose that $T$ is an linear operator and the following two estimates hold for all functions $f$:
$$ \|Tf\|_{L^{r_0}} \le C_0 \quad\mbox{and}\quad 
\|Tf\|_{L^{r_1}} \le C_1.$$
Then we have
$$\|Tf\|_{L^{r}} \le C_0^{1-\theta} C_1^\theta$$
for any $0\le \theta \le 1$ with
$$ \frac{1-\theta}{r_0} + \frac{\theta}{r_1}=\frac{1}{r} .$$ 
\item Suppose that $T$ is an linear operator and the following two estimates hold for all functions $f$:
\[||Tf||_{L^{r_0}}\le C_0 ||f||_{L^{p_0}},\quad||Tf||_{L^{r_1}}\le C_1 ||f||_{L^{p_1}}.\]
Then we have 
\[||Tf||_{L^r}\le C_0^{1-\theta}C_1^{\theta}||f||_{L^p},\]
for any $0\le \theta\le 1$ with \[\frac{1}{p}=\frac{1-\theta}{p_0}+\frac{\theta}{p_1}, \quad \frac{1}{r}=\frac{1-\theta}{r_0}+\frac{\theta}{r_1}.\]
\end{enumerate}
\end{proposition}
For $t\in \mathbb{F}_q$, let $\mu_3(t)$ be the number of triples $(x, y, z)\in A^3$ such that $||x+y+z||=t$. We recall the following lemma from \cite{covert2}. 
\begin{lemma}[\cite{covert2}, Lemma 2.6, Lemma 2.7]\label{loai0}
Let $A\subset\mathbb{F}_q^d$ with $d\ge 4$ even. If  $|A|\ge 3q^{d/2}$, then we have 
\[\left(|A|^3-\mu_3(0)\right)^2\ge \frac{|A|^6}{9},\]
and 
\[q^{-d}\left\vert \sum_{x\in S_0}\left(\widehat{A}(x)\right)^3  \right\vert^2 -\mu_3(0)^2\le \frac{4|A|^6}{q}.\]
Here, we note that $\widehat{A}(x)=\sum_{m\in A}\chi(-x\cdot m).$ 
\end{lemma}

We have the following lemma on the average of the second moment of function $\mu_3(t)$.

\begin{lemma}\label{bo6}
Let $A$ be a set in $\mathbb{F}_q^d$ with $d$ even. 
\begin{enumerate}
\item Suppose that $|A|\ge 3 q^{d/2}$ and $d=4$, then we have 
\[\sum_{t\in \mathbb{F}_q^*}\mu_3(t)^2\ll \frac{|A|^6}{q}+q^3|A|^{\frac{13}{3}}.\]
\item Suppose that $|A|\ge 3q^{d/2}$ and $d\ge 6$, then we have 
\[\sum_{t\in \mathbb{F}_q^*}\mu_3(t)^2\ll \frac{|A|^6}{q}+q^{\frac{3d^2-5d+2}{4d-8}}|A|^{5-\frac{d}{2d-4}}.\]
\end{enumerate}
\end{lemma}
Now we will make a reduction for the proof of Lemma \ref{bo6}.
We observe that
\begin{align*}
\mu_3(t)&=\sum_{x, y, z\in \mathbb{F}_q^d}A(x)A(y)A(z)S_t(x+y+z)\\
&=\frac{1}{q^d}\cdot \sum_{x, y, z\in \mathbb{F}_q^d}A(x)A(y)A(z)\sum_{m\in \mathbb{F}_q^d}\widehat{S_t}(m)\chi(m\cdot (x+y+z))\\
&=\frac{1}{q^d}\cdot\sum_{m\in \mathbb{F}_q^d}\widehat{S_t}(m)\left(\overline{\widehat{A}(m)}\right)^3.
\end{align*}
Thus, we have 
\begin{align}\label{mot}
\sum_{t\in \mathbb{F}_q}\mu_3(t)^2=\sum_{t\in \mathbb{F}_q}\mu_3(t)\overline{\mu_3(t)}=\frac{1}{q^{2d}}\cdot\sum_{m, v\in \mathbb{F}_q^d}\sum_{t\in \mathbb{F}_q}\widehat{S_t}(m)\overline{\widehat{S_t}(v)} \left(\overline{\widehat{A}(m)}\right)^3\left({\widehat{A}(v)}\right)^3.
\end{align}
We also make use of the following lemma which is taken from \cite[Proposition 2.2]{sun}. 
\begin{lemma}
For $m, v\in \mathbb{F}_q^d$, we have 
\[\sum_{t\in \mathbb{F}_q}\widehat{S_t}(m)\overline{\widehat{S_t}(v)}=q^{2d}\cdot \bigg(\frac{\delta_0(m)\delta_0(v)}{q}+q^{-(d+1)}\sum_{s\ne 0}\chi(s\cdot (||m||-||v||))\bigg),\]
where $\delta_0(m)=1$ if $m$ is a zero vector, and  $0$ otherwise.
\end{lemma}
 
Substituting this bound to (\ref{mot}), we get
\begin{equation*}
\sum_{t\in \mathbb{F}_q}\mu_3(t)^2
=\frac{1}{q}\left|{\widehat{A}(0)}\right|^6+\frac{1}{q^{d+1}}\cdot \sum_{m, v\in \mathbb F_q^d}\left(\overline{\widehat{A}(m)}\right)^3\left({\widehat{A}(v)}\right)^3\left(\sum_{s\in \mathbb F_q}\chi(s(||m||-||v||))-1\right).\end{equation*}
Using the orthogonality of $\chi$,  we get
$$\sum_{t\in \mathbb{F}_q}\mu_3(t)^2=\frac{|A|^6}{q}+\frac{1}{q^d}\cdot\sum_{||m||=||v||}\left(\overline{\widehat{A}(m)}\right)^3 \left({\widehat{A}(v)}\right)^3-\frac{1}{q^{d+1}} \left\vert\sum_{v\in \mathbb{F}_q^d}\left({\widehat{A}(v)}\right)^3\right\vert^2.$$
Since the last term above is negative,  it follows that
\begin{align*}\sum_{t\in \mathbb{F}_q}\mu_3(t)^2 &\le \frac{|A|^6}{q}+\frac{1}{q^d}\cdot\sum_{||m||=||v||}\left(\overline{\widehat{A}(m)}\right)^3\left({\widehat{A}(v)}\right)^3\\
&\le \frac{|A|^6}{q}+\frac{1}{q^d}\cdot\sum_{r\in \mathbb{F}_q}\left\vert \sum_{||v||=r}\left({\widehat{A}(v)}\right)^3\right\vert^2.\\
\end{align*}
Lemma \ref{loai0} tells us that 
\begin{align*}
\sum_{t\ne 0}\mu_3(t)^2&\le \frac{|A|^6}{q}+\frac{1}{q^d}\cdot \left(\max_{r\ne 0}\sum_{v\in S_r}\left|{\widehat{A}(v)}\right|^3\right)\sum_{v\in \mathbb F_q^d}\left|{\widehat{A}(v)}\right|^3\\
&\le \frac{|A|^6}{q}+|A|^{2}\left(\max_{r\ne 0}\sum_{v\in S_r}\left|{\widehat{A}(v)}\right|^3\right),
\end{align*}
 where  we have used the H\"{o}lder inequality and the facts that $\sum_{v\in \mathbb F_q^d}|\widehat{A}(v)|^2= q^d |A|$,~ $|\widehat{A}(v)|\le |\widehat{A}(0)|=|A|$ to bound the sum $\sum_{v\in \mathbb F_q^d}\left|{\widehat{A}(v)}\right|^3$. 
 
Therefore, in order to prove Lemma \ref{bo6}, it is enough to address the following theorem which will be shown using $L^p\to L^4$ spherical restriction estimates. 

 \begin{theorem}
For $A\subset \mathbb{F}_q^d$ with $d$ even, the following statements hold. 
\begin{enumerate}
\item Suppose that $d=4$, then we have 
\[||\widehat{A}||_{L^3(S_t, d\sigma)}\ll |A|^{7/9} \quad \mbox{for any}~~  t\ne 0.\]
\item Suppose that $d\ge 6$, then we have  
\[||\widehat{A}||_{L^3(S_t, d\sigma)}\ll q^{-\frac{d^2-7d+6}{12(d-2)}}|A|^{1-\frac{d}{6d-12}} \quad \mbox{for any}~~  t\ne 0.\]
\end{enumerate}

 \end{theorem}
 \begin{proof}
 {\bf Case $1$:} For any function $f\colon S_t\to \mathbb{C}$, we recall the trivial bound 
 \[||(fd\sigma)^\vee||_{L^\infty(\mathbb{F}_q^4, dc)}\ll ||f||_{L^1(S_t, d\sigma)}.\]
 It was also proved in \cite[Theorem 1.5]{pham} that 
 \[||(fd\sigma)^\vee||_{L^4(\mathbb{F}_q^4, dc)}\ll ||f||_{L^{\frac{8}{5}}(S_t, d\sigma)}.\]
Using Proposition \ref{co2}, we have 
 \[||(fd\sigma)^\vee||_{L^\frac{9}{2}(\mathbb{F}_q^4, dc)}\ll ||f||_{L^{\frac{3}{2}}(S_t, d\sigma)}.\]
 By duality, one has 
 \[||\widehat{g}||_{L^3(S_t, d\sigma)}\ll ||g||_{L^{\frac{9}{7}}(\mathbb{F}_q^4, dc)},\]
 for all function $g\colon \mathbb{F}_q^4\to \mathbb{C}$. Set $g$ to be the characteristic function of $A$, then the statement follows. 
 
{\bf Case $2$:}
To prove this case, we first need to show that 
 \[||\widehat{A}||_{L^2(S_t, d\sigma)}\ll \frac{|A|}{q^{\frac{d-1}{4}}},\]
whenever $|A|\ge q^{\frac{d-1}{2}}$.
Indeed, \[||\widehat{A}||^2_{L^2(S_t, d\sigma)}=\frac{1}{|S_t|}\sum_{x\in S_t}|\widehat{A}(x)|^2\sim\frac{1}{q^{d-1}}\sum_{x\in S_t}|\widehat{A}(x)|^2.\]
Thus, it is enough to handle the following inequality
\[\sum_{x\in S_t}|\widehat{A}(x)|^2\ll q^{\frac{d-1}{2}}|A|^2.\]
It follows from the definition of $\widehat{A}(x)$ that 
\begin{align*}
\sum_{x\in S_t}|\widehat{A}(x)|^2&=\sum_{x\in S_t}\sum_{a, b\in A}\chi(-x(a-b))=\sum_{a, b\in A}\widehat{S_t}(a-b)\\
&=|A|~\widehat{S_t}(0)+\sum_{a, b\in A, a\ne b}\widehat{S_t}(a-b)\\
&\le |A||S_t|+\sum_{a, b\in A, a\ne b}\left(\max_{x\ne 0}|\widehat{S_t}(x)|\right).
\end{align*}
Moreover, it was shown in \cite{IR06} that 
\[\max_{x\ne 0}|\widehat{S_t}(x)|=\max_{x\ne 0}\left\vert \sum_{m\in S_t}\chi(-x\cdot m)\right\vert\ll q^{\frac{d-1}{2}}.\]
Thus, we obtain 
\[\sum_{x\in S_t}|\widehat{A}(x)|^2\ll q^{d-1}|A|+q^{\frac{d-1}{2}}|A|^2\ll q^{\frac{d-1}{2}}|A|^2,\]
under the condition $|A|\ge q^{\frac{d-1}{2}}$. 

Since $d$ is even, by duality, it was proved in \cite[Theorem 1.5]{pham} that 
\[||\widehat{A}||_{L^{\frac{4d}{d+2}}(S_t, d\sigma)}\ll \|A\|_{L^{\frac{4}{3}}(\mathbb F_q^d, dc)}= |A|^{3/4}. \]
Thus, if $d\ge 2$, and $|A|\ge q^{\frac{d-1}{2}}$, using Proposition \ref{co2} with $\theta=2d/(3d-6)$, we obtain
\[||\widehat{A}||_{L^3(S_t, d\sigma)}\ll \left(q^{-\frac{d-1}{4}}|A|\right)^{1-\theta}|A|^{\frac{3\theta}{4}}.\]
Notice that since $d\ge 6$, we have $0\le\theta\le 1$. Hence,
\[||\widehat{A}||_{L^3(S_t, d\sigma)}\ll q^{-\frac{d^2-7d+6}{12(d-2)}}|A|^{1-\frac{d}{6d-12}}.\]
This completes the proof of the theorem.
 \end{proof}

\paragraph{Proof of Theorem \ref{thm:main4-thang}:}
Using the Cauchy-Schwarz inequality and the first statement of Lemma \ref{loai0}, we have 
\[|\Delta_3(A)|\gg \frac{|A|^6}{\sum_{t\ne 0}\mu_3(t)^2}.\]
On the other hand, Lemma \ref{bo6} gives us 
\[\sum_{t\ne 0}\mu_3(t)^2\ll \begin{cases}&\frac{|A|^6}{q}+q^3|A|^{\frac{13}{3}}, ~d=4\\ &\frac{|A|^6}{q}+q^{\frac{3d^2-5d+2}{4d-8}}|A|^{5-\frac{d}{2d-4}}, ~d\ge 6. \end{cases}\]
In other words, 
\[|\Delta_3(A)|\gg q,\]
whenever \[|A|\gg \begin{cases} q^{\frac{12}{5}}, ~d=4\\ q^{\frac{d+1}{2}-\frac{1}{3d-4}}, ~d\ge 6,\end{cases}\] This completes the proof of the theorem. $\square$

\section*{Acknowledgements}

The authors gratefully thank to the refree for the valuable comments which helped us improve the quality of our manuscript. 

D. Koh was supported by Basic Science Research Program through the National
Research Foundation of Korea(NRF) funded by the Ministry of Education, Science
and Technology (NRF-2018R1D1A1B07044469). T. Pham was supported by Swiss National Science Foundation grants P400P2--183916 and P4P4P2-191067.

The authors would like to thank to the VIASM for the hospitality and for the excellent working condition.
 \bibliographystyle{amsplain}

\begin{thebibliography}{10}
\bibitem{CEHIK10} J. Chapman, M. Burak Erdo\u{g}an, D. Hart, A. Iosevich, and D. Koh, {\it Pinned distance sets, k-simplices, Wolff's exponent in finite fields and sum-product estimates}, Math Z. \textbf{271}(1) (2012), 63--93. 

\bibitem{covert2}
D. Covert, D. Koh, and Y. Pi, \textit{On the sums of any $k$ points in finite fields}, SIAM J. Discrete Math, \textbf{30}(1) (2016), 367--382. 

\bibitem{covert} D. Covert, D. Koh, and Y. Pi, \textit{The generalized $k$-resultant modulus set problem in finite fields}, J. Fourier Anal. Appl. \textbf{25}(3) (2019), 1026--1052. 

\bibitem{falconer}
K. J. Falconer, \textit{On the Hausdorff dimensions of distance sets}, Mathematika, \textbf{32}(2) (1985),  206--212.

\bibitem{aka}
L. Grafakos, \textit{Classical and modern Fourier analysis}, Pearson. Education, Inc. (2004).

\bibitem{line}
W. H. Haemers, \textit{Eigenvalue techniques in design and graph theory}, 
Dissertation, Technische Hogeschool Eindhoven, Eindhoven, 1979. Mathematical Centre Tracts, 121. Mathematisch Centrum, Amsterdam, 1980. v+102 pp.

\bibitem{hart}
 D. Hart, A. Iosevich, D. Koh, and M. Rudnev, \textit{Averages over hyperplanes, sum-product theory in vector spaces over finite fields and the Erd\H{o}s-Falconer distance conjecture}, Trans. Amer. Math. Soc.  \textbf{363}(6)  (2011), 3255--3275.

\bibitem{IK08} A. Iosevich and D. Koh, \emph{Extension theorems for the Fourier transform associated with non-degenerate quadratic surfaces in vector spaces over finite fields}, Illinois J. of Mathematics, \textbf{52}(2) (2008), 611--628.

\bibitem{I-K}
 A. Iosevich and D. Koh, \textit{Extension theorems for spheres in the finite field setting}, Forum. Math. \textbf{22}(3) (2010), 457--483.
 
\bibitem{IKL17} A. Iosevich, D. Koh, and M. Lewko, \emph{Finite field restriction estimates for the paraboloid in high even dimensions}, J. Funct. Anal.  \textbf{278}(11) (2020), 108450.
 
\bibitem{pham}
A. Iosevich, D. Koh, S. Lee, T. Pham,  and C-Y. Shen, \textit{On restriction estimates for the zero radius sphere over finite fields},  Canad. J. Math.   \textbf{73}(3) (2021),  769--786.

\bibitem{IR06} A. Iosevich, M. Rudnev, \emph{Erd\H{o}s-Falconer  distance problem in vector spaces over finite fields}, Trans. Amer. Math. Soc. \textbf{359}(12)  (2007), 6127--6142.

\bibitem{KLP}
D. Koh, S. Lee, and T. Pham, \emph{On the cone restriction conjecture in four dimensions and applications in incidence geometry}, International Mathematics Research Notices, \textbf{2022}(21) (2022), 17079--17111. 
\bibitem{KPV18}
D. Koh, T. Pham, and L. A. Vinh, \emph{Extension theorems and a connection to the Erd\H{o}s-Falconer distance problem over finite fields}, J. Funct. Anal. \textbf{281}(8) (2021), 109137. 

\bibitem{sun}
D. Koh and  H. Sun, \textit{Distance sets of two subsets of vector spaces over finite fields}, Proc. Amer. Math. Soc. \textbf{143}(4) (2015), 1679--1692. 


\bibitem{Le13}  M. Lewko, \emph{New restriction estimates for the $3$-d paraboloid over finite fields}, Adv. Math. {\bf 270}(1) (2015), 457--479.

\bibitem{Le19} M. Lewko, \emph{Finite field restriction estimates based on Kakeya maximal operator estimates,} J. Eur. Math. Soc.  \textbf{21}(12)  (2019),  3649--3707. 

\bibitem{Le20} M. Lewko, \emph{Counting rectangles and an improved restriction estimate for the paraboloid in $\mathbb F_p^3$},  Proc. Amer. Math. Soc.  \textbf{148}(4) (2020), 1535--1543.

\bibitem{LN97} R. Lidl and H. Niederreiter, \emph{Finite fields}, Cambridge University Press, (1997).

\bibitem{MT04} G. Mockenhaupt and T. Tao, \emph{Restriction and Kakeya phenomena for finite fields}, Duke Math. J. \textbf{121}(1) (2004), 35--74.

\bibitem{mu}
B. Murphy, G. Petridis, T.  Pham, M. Rudnev,  and S. Stevens, \emph{On the pinned distances problem in positive characteristic},  J. Lond. Math. Soc. \textbf{105}(1) (2022), 469--499.

 
\bibitem{RS19} M. Rudnev and I. Shkredov, \emph{On the restriction problem for discrete paraboloid in lower dimension}, Adv. Math. \textbf{339} (2018), 657--671.


\bibitem{vinhline}
L. A. Vinh, \emph{The Szemer\'{e}di–Trotter type theorem and the sum-product estimate in finite fields},
European J. Combin. \textbf{32}(8) (2011), 1177--1181.

 

\end{thebibliography}

\end{document}